\documentclass[11pt, a4paper, notitlepage]{article}

\usepackage{amsmath, latexsym, amsthm, amsfonts} 
\usepackage{graphicx} 
\usepackage{epsfig}
\usepackage{varioref}
\usepackage{natbib} 
\usepackage{setspace} 

\bibpunct{(}{)}{;}{a}{}{,} 

\theoremstyle{plain}
\newtheorem{theorem}{Theorem}
\newtheorem{cnd}{Condition}
\newtheorem{lemma}{Lemma}

\newcommand{\infint}{\int_{-\infty}^{\infty}}

\def\convd{\stackrel{\cal D}{\rightarrow}}
\def\convp{\stackrel{\operatorname{P}}{\rightarrow}}
\def\eqd{\stackrel{\operatorname{\cal D}}{=}}
\def\ex{{\rm E\,}}
\def\var{\mathop{\rm Var}\nolimits}
\def\cov{\mathop{\rm Cov}\nolimits}

\begin{document}
\title{Weak convergence of the supremum distance for supersmooth  kernel deconvolution}
\author{Bert van Es\\
{\normalsize Korteweg-de Vries Institute for Mathematics}\\
{\normalsize Universiteit van Amsterdam}\\
{\normalsize Plantage Muidergracht 24}\\
{\normalsize 1018 TV Amsterdam}\\
{\normalsize The Netherlands}\\
{\normalsize vanes@science.uva.nl}\\
{}\\
{\normalsize Shota Gugushvili}\\
{\normalsize Eurandom}\\
{\normalsize Technische Universiteit Eindhoven}\\
{\normalsize P.O.\ Box 513}\\
{\normalsize 5600 MB Eindhoven}\\
{\normalsize The Netherlands}\\
{\normalsize gugushvili@eurandom.tue.nl}
}
\maketitle
\begin{abstract}
We derive the asymptotic distribution of the supremum distance of
the deconvolution kernel density estimator to its expectation for
certain supersmooth deconvolution problems. It turns out  that the
asymptotics are essentially different from the corresponding results
for   ordinary smooth deconvolution.
\end{abstract}
\newpage


\section{Introduction and results}

Consider the classical deconvolution problem:
let $X_1,\ldots, X_n$ be i.i.d.\ observations, where $ X_i=Y_i+Z_i
$ and $Y_i$ and $Z_i$ are independent. Assume that the
unobservable $Y_i$ have distribution function $F$ and density $f$,
and that the random variables $Z_i$ have a known density $k$. Note
that the density $g$ of $X_i$ is equal to the convolution of $f$
and $k$. The nonparametric deconvolution problem is the problem
of estimating $f$ or $F$ from the observations $X_i$. Thus we want
to recover the distribution of $Y_i$ using the contaminated
measurements $X_i.$ Additional information on measurement error models and many practical examples can be found in \citet{carroll3}.

A popular density estimator for this problem is the deconvolution kernel density estimator introduced in \citet{carroll1} and \citet{carroll2}. This estimator is defined as
\begin{equation}\label{fnh}
f_{nh}(x)=\frac{1}{2\pi} \int_{-\infty}^\infty
e^{-itx} \frac{\phi_w(ht)\phi_{emp}(t)}{ \phi_k(t)}\,dt
=
\frac{1}{nh}\sum_{j=1}^n v_h\Big(\frac{x-X_j}{h}\Big),
\end{equation}
with
$$
v_h(u)=\frac{1}{2\pi}\infint \frac{\phi_w(s)}{\phi_k(s/h)}\
e^{-isu}ds.
$$
Here $w$ denotes a {\em kernel function}, $h>0$ is a {\em bandwidth}, $\phi_{emp}$ is the empirical characteristic function of the sample defined by $ \phi_{emp}(t) = (1/n)\sum_{j=1}^n
e^{itX_j},$ and $\phi_w$ and $\phi_k$ denote the characteristic
functions of $w$ and $k,$ respectively. Note that \eqref{fnh} is
not a standard kernel density estimator, because the kernel
function $v_h$ depends on the bandwidth $h.$ For an introduction to the estimator \eqref{fnh} see e.g.\ \citet{jones}.

The rate of decay to zero at minus and plus infinity of the
modulus of the characteristic function $\phi_k,$ and consequently
the smoothness of $k,$ is crucial to the asymptotic behaviour of
\eqref{fnh}. Two cases have been distinguished, the {\em ordinary
smooth case}, where $|\phi_k|$ decays algebraically to zero, and
the {\em supersmooth case}, where it decreases exponentially. The
asymptotics in the ordinary smooth case are essentially the same
as for a kernel estimator of a higher order derivative of a
density, see e.g.\ \citet{fan2}, \citet{fan4} and \citet{vanes0}.
The asymptotics in the supersmooth case have been studied e.g.\ in
\citet{fan2} and \citet{vanes2,vanes1}.


Notice that the above papers study local properties of the estimator \eqref{fnh},
i.e.\ its pointwise behaviour. We, on the other hand, will focus on the asymptotic
behaviour of the supremum distance of the estimator to its expectation, which
provides a global measure of its performance. Accordingly, define
\begin{equation}\label{supdist}
M_n = \sup_{0\leq x\leq 1} |f_{nh}(x) - \ex [f_{nh}(x)]|.
\end{equation}
The fact that the supermum is taken over $[0,1]$ is not a
restriction of generality and is for convenience only. One could
have considered any interval $[a,b].$ An alternative here is to
consider the integrated squared error of the estimator $f_{nh}$.
This was done in \citet{holzmann}.

The asymptotic distribution of the supremum distance similar to
\eqref{supdist}, namely $
\sup_{x\in[0,1]}(g(x))^{-1/2}|g_{nh}(x)-\ex[g_{nh}(x)]|, $ for an
ordinary kernel density estimator $g_{nh}$ in the direct density
estimation setting (i.e.\ in the error-free case) was derived in
\citet{bickelr}. Owing in a certain sense to the similarity of the
asymptotics in the ordinary smooth deconvolution problem to that
in the direct density estimation problem, qualitatively similar
results were obtained in \citet{bissantz} in the ordinary smooth
deconvolution problem for the supremum distance
$\sup_{x\in[0,1]}(g(x))^{-1/2}|f_{nh}(x)-\ex[f_{nh}(x)]|.$
Normalisation with $\sqrt{g(x)}$ is explainable by the fact that
the expression for the asymptotic variance in the asymptotic
normality theorem for the estimator $f_{nh}(x)$ in the ordinary
smooth deconvolution problem involves $g(x),$ see \citet{fan2}. No
direct extension of the methods used in \citet{bickelr} to the
supersmooth deconvolution problem is possible and derivation of
the asymptotic distribution of \eqref{supdist} requires a
different approach. This is precisely the task of the present
paper. Notice that in \eqref{supdist} we do not have to normalise
with $\sqrt{g(x)},$ because the asymptotic variance in the
asymptotic normality theorem for this case does not depend on $g,$
but only on the error density $k$ (in some global way), see
\citet{vanes1}.

We now state the conditions on the density
$k$ and kernel $w,$ which will be used throughout the paper. The
condition on $k$ which defines supersmooth deconvolution is given
in Condition \ref{condk}.
%
%

\begin{cnd}
\label{condk} Assume that
\begin{equation}
\label{chk} \phi_k(t)= C
|t|^{\lambda_0}\exp\left[-|t|^{\lambda}/\mu\right](1+o(|t|^{-1}))
\end{equation}
as $|t|\rightarrow \infty,$ for a constant $0<\lambda\leq 2$ and
some constants $\mu>0,\lambda_0\in\mathbb{R}$ and
$C\in\mathbb{R}.$ Furthermore, let $\phi_k(t)\neq 0$ for all
$t\in{\mathbb{R}}.$
\end{cnd}

Condition \ref{condk} is stronger than the usual condition on $k$
in supersmooth deconvolution given e.g.\ in \citet{vanes1}, where
the term $o(|t|^{-1})$ is not present and one just has the
asymptotic equivalence.

\begin{cnd}
\label{condw}
Let $\phi_w$ be real-valued, symmetric and have support $[-1,1].$ Let $\phi_w(0)=1,$ and
assume $\phi_w(1-t)=At^{\alpha}+o(t^{\alpha})$ as $t\downarrow 0$ for some constants $A$ and $\alpha\geq 0.$
\end{cnd}
For examples of such kernels see for instance \cite{vanes1}.

The next theorem establishes the asymptotic distribution of $M_n$,
which could prove useful for the construction of uniform
confidence bands around $f$. Since it will appear repeatedly in
the paper, we will write $\zeta(h)$ for $\exp( 1/(\mu
h^\lambda))$.
\begin{theorem}\label{thm:sup}
Assume Condition \ref{condk} for $\lambda=2$ and Condition
\ref{condw} and let $\ex [X_j^2]<\infty$. Let $V$ denote a
positive random variable with a Rayleigh distribution with density
$f_V(x)=x\exp[{-x^2/2}]I_{[x\geq 0]}.$ Then, as $n\to\infty$ and
$h\to 0$,
\begin{equation}
\label{mnan} \frac{\sqrt{n}}{h^{\lambda(1+\alpha)+\lambda_0-1}
\zeta(h)}\, M_n\convd \frac{1}{2}\sqrt{2}\,\frac{A}{\pi C}\,\Big({\mu \over
\lambda}\Big)^{1+\alpha}\Gamma(\alpha
+1)\,V,
\end{equation}
where $\Gamma$ denotes the gamma function.
\end{theorem}

By assuming $\lambda=2$ we restrict ourselves to deconvolution
problems for error distributions with characteristic functions
that have an exponential tail like the characteristic function of
a normal density. The most important case covered by this
condition is standard normal deconvolution, where $\lambda=2,
\lambda_0=0,\mu=2$ and $C=1.$ The condition $\lambda=2$ seems to
be essential in the proof of Lemma \ref{lemma3}, specifically in
\eqref{bound}, where we prove a condition for tightness of the
remainder process $R^{(1)}_n$. Whether it can be relaxed by other
approaches, avoiding tightness, remains open.

The rate of convergence in  Theorem \ref{thm:sup} once again
reflects the difficulty of the supersmooth deconvolution problem
compared to the ordinary smooth deconvolution. Furthermore, unlike
in ordinary smooth deconvolution, see \citet{bissantz}, in order
to obtain the asymptotic distribution of $M_n,$ we do not have to
subtract a drift term. This also has a parallel when considering
the asymptotics of the $\operatorname{ISE}[f_{nh}]$ in the
supersmooth deconvolution, see \citet{holzmann} for additional
details. Notice also that unlike the direct density estimation or
the ordinary smooth deconvolution, see \citet{bickelr} and
\citet{bissantz}, the limit distribution in \eqref{mnan} is not
Gumbel, which confirms the conjecture in \citet{bissantz} for the
case $\lambda=2.$

\section{Proof of Theorem \ref{thm:sup}}

The proof of Theorem \ref{thm:sup} is based on a decomposition of $f_{nh}(x)$ in \citet{vanes1},
which is the basis of the proof of their asymptotic normality theorem. We have
\begin{equation}
\label{dec}
\begin{split}
f_{nh}(x)&=
\frac{1}{\pi C}\,h^{\lambda_0-1}\int_\epsilon^1 \phi_w(s)
s^{-\lambda_0} \exp(s^\lambda/(\mu h^\lambda))ds  \frac{1}{n}\sum_{j=1}^n \cos\Big(\frac{X_j-x}{h}\Big)\\
&+
R_n^{(1)}(x)+R_n^{(2)}(x)+R_n^{(3)}(x),
\end{split}
\end{equation}
where
$
R_n^{(l)}(x)=({1}/{n})\sum_{j=1}^n R^{(1)}_{n,j}(x), l=1,2,3,
$
and
\begin{align*}
&R^{(1)}_{n,j}(x)=\frac{1}{C}{1\over \pi}h^{\lambda_0-1}\int_\epsilon^1
\Big(\cos\Big(s\Big(\frac{X_j-x}{h}\Big)\Big)- \cos\Big(\frac{X_j-x}{h}\Big)\Big)\nonumber\\
&\quad\quad\quad\quad \times \phi_w(s)s^{-\lambda_0} \exp(s^{\lambda}/(\mu h^\lambda))ds\\
&R_{n,j}^{(2)}(x)=\frac{1}{2\pi h}\int_{-\epsilon}^\epsilon
\exp\Big(is\Big(\frac{X_j-x}{h}\Big)\Big)\phi_w(s)\frac{1}{\phi_k(s/h)}ds\\
&R_{n,j}^{(3)}(x)=\frac{1}{2\pi h}
\Big(\int_{-1}^{-\epsilon}+\int_{\epsilon}^1\Big)
\exp\Big(is\Big(\frac{X_j-x}{h}\Big)\Big)\phi_w(s)\nonumber\\
&\quad\quad\quad\quad\times\Big(\frac{1}{\phi_k(s/h)}-\frac{1}{C}\Big(\frac{|s|}{h}\Big)^{-\lambda_0}
\exp(|s|^\lambda/(\mu h^\lambda))\Big) ds.
\end{align*}
We will write $R_n^{(l)},l=1,2,3$ for the stochastic processes
$R_n^{(l)}=(R_n^{(l)}(x))_{x\in[0,1]}.$ Notice that these
processes belong to the space $C[0,1].$

Now the rough idea is to derive the asymptotic distribution of the
supremum of the first summand in \eqref{dec} minus its expectation
and to show that the remainder terms are negligible. Define the
process $U_n$ as $ U_n(x)=n^{-1/2}\, \sum_{j=1}^n U_{n,j}(x), $
where $ U_{n,j}(x)=\cos((X_j-x)/h)- \ex [\cos((X_j-x)/h)] $. Note
that this is a process with expectation equal to zero at every
$x$. Write
\begin{equation*}
S_n = \sup_{0\leq x\leq 1} |U_n(x)|.
\end{equation*}

\begin{lemma}\label{limsup}
Under the conditions of Theorem \ref{thm:sup} we have, as
$n\to\infty$ and $h\to 0$,
$$
 S_n\convd \sup_{0\leq x\leq 2\pi} |W(x)|,
$$
where $W$ is a zero mean Gaussian process   on $[0,2\pi]$ with
covariance function
$\cov(W(x_1),W(x_2))=({1}/{2})\cos(x_1-x_2)$.
\end{lemma}

\begin{proof}

Replacing $x$ by $yh,$ by the periodicity of
the cosine function we have for $h\leq (2\pi)^{-1}$ that
\begin{eqnarray*}
\lefteqn{S_n = \sup_{0\leq x\leq 1} |U_n(x)|}\\
&=& \sup_{0\leq x\leq 1} \Big|\frac{1}{\sqrt{n}}\, \sum_{j=1}^n
\Big(\cos\Big(\frac{X_j-x}{h}\Big)- \ex
\Big[\cos\Big(\frac{X_j-x}{h}\Big)\Big]\Big)\Big|\\
&=& \sup_{0\leq  y\leq 1/h} \Big|\frac{1}{\sqrt{n}}\, \sum_{j=1}^n
\Big(\cos\Big(\frac{X_j-yh}{h}\Big)- \ex
\Big[\cos\Big(\frac{X_j-yh}{h}\Big)\Big]\Big)\Big|\\
&=& \sup_{0\leq  y\leq 1/h} \Big|\frac{1}{\sqrt{n}}\, \sum_{j=1}^n
 (\cos(Y_j-y)- \ex
[\cos(Y_j-y)] )\Big|\\
&=& \sup_{0 \leq  y\leq 2\pi } \Big|\frac{1}{\sqrt{n}}\,
\sum_{j=1}^n  (\cos(Y_j-y)- \ex [\cos(Y_j-y)] )\Big|,\\
&=& \sup_{0\leq y\leq 2\pi} |W_n(y)|,
\end{eqnarray*}
where $ Y_j=(X_j/h)\negthickspace\negthickspace\mod 2\pi $ and the
process $W_n$ on $[0,2\pi]$ is given by $ W_n(y)=n^{-1/2}\,
\sum_{j=1}^n (W_{n,j}(y)-\ex[W_{n,j}(y)]) $ with $
W_{n,j}(y)=\cos(Y_j-y). $

By Lemma 6 of \citet{vanes1} we know that $Y_j\convd
\operatorname{Un}(0,2\pi)$ as $h\rightarrow 0$ for each $j,$ where
$\operatorname{Un}(0,2\pi)$ denotes the uniform distribution on
$[0,2\pi].$ Hence by the dominated convergence theorem we get that
\begin{eqnarray*}
\lefteqn{\cov \Big(\cos\Big( Y_j-y_1 \Big),\cos\Big( Y_j-y_2 \Big)\Big)} \\
&\rightarrow& \frac{1}{2\pi}\int_0^{2\pi}
\cos(u-y_1)\cos(u-y_2)du= \frac{1}{2}\cos(y_1-y_2).
\end{eqnarray*}

It follows that we have to study the convergence of the process
$W_n(x) -\ex [W_n(x)]$ which belongs to $C[0,2\pi].$ According to
Prohorov's theorem and in particular Theorem 8.1 of
\citet{billingsley}, it suffices to show weak convergence of the
finite dimensional distributions and tightness of the sequence. By
the multivariate central limit theorem in the triangular array
scheme or Cramer-Wold device, see Theorem 7.7 in \citet{billingsley}, the finite dimensional distributions of the process $W_n$
converge to multivariate normal distributions with covariances
given by $\cov(W(y_1),W(y_2))=({1}/{2})\cos(y_1-y_2)$. To prove
tightness, we will verify conditions of Theorem 12.3 of
\citet{billingsley}. First of all, notice that the sequence
$W_n(0)$ is tight, because the asymptotic normality of $W_n(0)$
follows by a univariate Lyapunov central limit theorem in a
trinagular array scheme, see Theorem 7.3 in \citet{billingsley}.
Furthermore, for an arbitrary positive $\eta,$
\begin{align*}
&P\Big(|W_n(y_2)-\ex [W_n(y_2)]-(W_n(y_1)-\ex [W_n(y_1)])|\geq\eta\Big)\\
&\leq \frac{1}{\eta^2}\, \var [W_{n,j}(y_2)-W_{n,j}(y_1)]\leq
\frac{1}{\eta^2}\,  \ex [(W_{n,j}(y_2)-W_{n,j}(y_1))^2]\\
&\leq \frac{1}{\eta^2}\,(y_2-y_1)^2,
\end{align*}
which follows from the fact that
\begin{align*}
|\cos(Y_j-y_2)-\cos(Y_j-y_1)|&=\left|2\sin\left(\frac{2Y_j-y_2-y_1}{2}\right)\sin\left(\frac{y_1-y_2}{2}\right)\right|\\
&\leq|y_1-y_2|.
\end{align*}
Here we used the inequality $|\sin x|\leq |x|.$ Therefore $W_n$
converges weakly to a zero mean Gaussian process $W$ on $[0,2\pi]$
with covariance function
$\cov(W(y_1),W(y_2))=({1}/{2})\cos(y_1-y_2).$ By the continuous
mapping theorem, see Theorem 5.1 in \citet{billingsley},
the supremum of $|W_n|$ then converges weakly to the supremum of
the absolute value of the limit process, which proves the lemma.
\end{proof}

\begin{lemma}\label{distsup}
\label{lemma2}
With $V$ as in Theorem \ref{thm:sup}, we have
\begin{equation}
\label{rayleigh}
\sup_{0\leq x\leq 2\pi} |W(x)|\eqd
\frac{1}{2}\sqrt{2}\,V.
\end{equation}
\end{lemma}

\begin{proof}
Let $N_1$ and $N_2$ denote two independent standard normal random
variables and let us define the process $\tilde W$ by $\tilde
W=(\tilde W(x))_{x\in[0,2\pi]}$, where
\begin{equation*}
\tilde W(x)\eqd\frac{1}{2}\sqrt{2}(N_1\cos x +N_2\sin x).
\end{equation*}
Since the covariance function $\cov(W(x_1),W(x_2))$ of the process
$W$, given by $({1}/{2})\cos(x_1-x_2)$,  equals $\cov(\tilde
W(x_1),\tilde W(x_2))$ by
\begin{align*}
\cov\Big(\frac{1}{2}\sqrt{2}&(N_1\cos x_1 +N_2\sin x_1),\frac{1}{2}\sqrt{2}(N_1\cos x_2 +N_2\sin x_2)\Big)\\
&=\frac{1}{2}\,(\cos x_1\cos x_2+ \sin x_1\sin x_2)=\frac{1}{2}\,\cos(x_1-x_2),
\end{align*}
it follows that $ W\eqd \tilde W$.

Next write
\begin{eqnarray}
\lefteqn{\frac{1}{2}\sqrt{2}(N_1\cos x +N_2\sin x)}\nonumber\\
&=&
\frac{1}{2}\sqrt{2}\sqrt{N_1^2+N_2^2}\Big(\frac{N_1}{\sqrt{N_1^2+N_2^2}}\,\cos
x +
\frac{N_2}{\sqrt{N_1^2+N_2^2}}\,\sin x\Big)\nonumber\\
&=&
\frac{1}{2}\sqrt{2}\sqrt{N_1^2+N_2^2}(\cos \xi\cos x+\sin \xi\sin x)\nonumber\\
&=&
\frac{1}{2}\sqrt{2}\sqrt{N_1^2+N_2^2}\cos(x-\xi),\label{shiftcos}
\end{eqnarray}
for a $\xi$ such that $\cos\xi =N_1/\sqrt{N_1^2+N_2^2}$ and
$\sin\xi = N_2/\sqrt{N_1^2+N_2^2}$. The supremum of the absolute
value of \eqref{shiftcos} is equal to
$({1}/{2})\sqrt{2}\sqrt{N_1^2+N_2^2}=({1}/{2})\sqrt{2}V,$ where
$V$ has a Rayleigh distribution. This entails \eqref{rayleigh}.
\end{proof}

\begin{lemma}
\label{lemma3} Let
$a_n={\sqrt{n}}{h^{-\lambda(1+\alpha)-\lambda_0+1}
(\zeta(h))^{-1}}$ denote the normalising sequence in Theorem
\ref{thm:sup}. For $l=1,2,3$ we have
\begin{equation*}
a_n(R_n^{(l)}-\ex [R_n^{(l)}])\convp {\bf{0}}
\end{equation*}
as $n\rightarrow\infty$ and $h\rightarrow 0.$ Here ${\bf{0}}$
denotes the zero process on $[0,1].$
\end{lemma}

\begin{proof}
To prove the lemma, we will apply Prohorov's theorem, and in
particular Theorem 8.1 of \citet{billingsley}. Firstly, notice
that for a fixed $x$ the remainder terms $a_n(R_n^{(l)}(x)-\ex
[R_n^{(l)}(x)])$ vanish in probability, which was proved in
\citet{vanes1}. This implies that the finite dimensional vectors
of the processes $a_n(R_n^{(l)}-\ex [R_n^{(l)}])$ also converge in
probability to null vectors. To establish tightness, we will again
verify conditions of Theorem 12.3 of \citet{billingsley}. Notice
that when $x=0,$ the sequence $a_n(R_n^{(l)}(0)-\ex
[R_n^{(l)}(0)])$ is tight, since it converges to zero in
probability. Furthermore, for an arbitrary positive $\eta$ we have
\begin{align}
P\Big(a_n|R_n^{(1)}&(x_2)-\ex [R_n^{(1)}(x_2)] -(R_n^{(1)}(x_1)-\ex [R_n^{(1)}(x_1)])|\geq\eta\Big)\nonumber\\
&\leq
\frac{a_n^2}{\eta^2}\, \var [R_n^{(1)}(x_2)-R_n^{(1)}(x_1)]\nonumber\\
&=
\frac{a_n^2}{\eta^2}\,\frac{1}{n}\, \var [R_{n,1}^{(1)}(x_2)-R_{n,1}^{(1)}(x_1)]\nonumber\\
&\leq
\frac{a_n^2}{\eta^2}\,\frac{1}{n}\, \ex [(R_{n,1}^{(1)}(x_2)-R_{n,1}^{(1)}(x_1))^2]\nonumber\\
&\leq
\frac{a_n^2}{\eta^2}\,\frac{1}{C^2}{1\over \pi^2}\frac{1}{n}\,h^{2(\lambda_0-1)}(x_2-x_1)^2\nonumber\\
&\quad\quad\quad\quad\times K^2\Big(\int_\epsilon^1
\Big(\frac{1-s}{h^2}\Big)\phi_w(s)s^{-\lambda_0}
\exp(s^{\lambda}/(\mu h^\lambda))ds\Big)^2\nonumber\\
&=
K^2\frac{a_n^2}{\eta^2}\,\frac{1}{C^2}{1\over \pi^2}\frac{1}{n}\,h^{2(\lambda_0-2)-2}(x_2-x_1)^2\nonumber\\
&\quad\quad\quad\quad\times\Big(\int_\epsilon^1
\Big(1-s)\phi_w(s)s^{-\lambda_0}
\exp(s^{\lambda}/(\mu h^\lambda))ds\Big)^2\nonumber\\
&= O\Big(\frac{1}{n}\,h^{2(\lambda_0-2)-2+2(2+\alpha)\lambda}(\zeta(h))^2a_n^2\Big)\frac{1}{\eta^2}(x_2-x_1)^2\nonumber\\
&= O\Big(h^{2(\lambda -2)}\Big)\frac{1}{\eta^2}\,(x_2-x_1)^2\label{bound}\\
&= O(1)\frac{1}{\eta^2}\,(x_2-x_1)^2.\nonumber
\end{align} where $K$ is some constant. Here we used Lemma 5 of \citet{vanes1}, which states that
\begin{eqnarray}
\lefteqn{\int_\epsilon^1 s^{-\lambda_0}(1-s)^\beta \phi_w(s)
\exp(s^\lambda/(\mu h^\lambda))ds}\nonumber\\
&\qquad\qquad\sim& A \Big({\mu \over
\lambda}h^{\lambda}\Big)^{1+\alpha+\beta} \zeta(h)
\Gamma(\alpha+\beta+1),\label{eu}
\end{eqnarray}
and the fact that for $0\leq s \leq 1$ and $0\leq x_1<x_2\leq 1$
we have

\begin{align*}
&\Big|\cos\Big(s\Big(\frac{X_j-x_2}{h}\Big)\Big)-\cos\Big(\frac{X_j-x_2}{h}\Big)
-\cos\Big(s\Big(\frac{X_j-x_1}{h}\Big)\Big)+\cos\Big(\Big(\frac{X_j-x_1}{h}\Big)\Big)\Big|\\
&\quad\quad= \Big|\int_{x_1}^{x_2}\int_s^1
\frac{\partial^2}{\partial u\partial v}\Big\{
\cos\Big(v\Big(\frac{X_j-u}{h}\Big)\Big)-\cos\Big(\frac{X_j-u}{h}\Big)\Big\}dudv\Big|\\
&\quad\quad= \Big|\int_{x_1}^{x_2}\int_s^1 \Big\{\frac{1}{h}
\sin\Big(v\Big(\frac{X_j-u}{h}\Big)\Big)+v\Big(\frac{X_j-u}{h^2}\Big)
\cos\Big(\frac{X_j-u}{h}\Big)\Big\}dudv\Big|\\
&\quad\quad\leq
\int_{x_1}^{x_2}\int_s^1 \frac{1}{h^2}\,(|X_j|+1+h)\, dudv\\
&\quad\quad\leq \frac{1}{h^2}\,(|X_j|+1+h)(1-s)|x_1-x_2| .
\end{align*}
Hence the process $a_n(R_n^{(1)}-\ex
[R_n^{(1)}])$ is tight.

In order to prove tightness of the process $a_n(R_n^{(2)}-\ex
[R_n^{(2)}]),$ note that, as above, for positive $\eta$
\begin{align*}
P\Big(a_n|R_n^{(2)}&(x_2)-\ex [R_n^{(2)}(x_2)] -(R_n^{(2)}(x_1)-\ex [R_n^{(2)}(x_1)])|\geq\eta\Big)\\
&\leq
\frac{a_n^2}{\eta^2}\,\frac{1}{n}\, \ex [(R_{n,1}^{(2)}(x_2)-R_{n,1}^{(2)}(x_1))^2]\\
&\leq
4\frac{a_n^2}{\eta^2}\,{1\over 4\pi^2h^2}\,\frac{1}{n}
\Big(\int_{-\epsilon}^\epsilon \frac{s}{h}\phi_w(s)\frac{1}{\phi_k(s/h)}\,ds\Big)^2(x_2-x_1)^2\\
&=
4\frac{a_n^2}{\eta^2}\,{1\over 4\pi^2h^4}\,\frac{1}{n}
\Big(\int_{-\epsilon}^\epsilon s\phi_w(s)\frac{1}{\phi_k(s/h)}\,ds\Big)^2(x_2-x_1)^2\\
&\leq
4 a_n^2{1\over 4\pi^2h^4}\,\frac{1}{n}\, (2\epsilon)^2\epsilon^2
\Big(\int_{-\epsilon}^\epsilon \frac{1}{\phi_k(s/h)}\,ds\Big)^2
\frac{1}{\eta^2}(x_2-x_1)^2\\
&\leq 4 a_n^2{1\over \pi^2h^4}\,\frac{1}{n}\, \epsilon^4
\Big(\sup_{-\epsilon\leq s\leq\epsilon}
\frac{1}{\phi_k(s/h)}\Big)^2
\frac{1}{\eta^2}(x_2-x_1)^2\\
&\leq
4a_n^2{2\over C^2\pi^2 }\,\frac{1}{n}\,
(\epsilon/h)^{4-2\lambda_0}\exp(2(\epsilon/h)^{\lambda}/\mu)
\frac{1}{\eta^2}(x_2-x_1)^2\\
&=
o(1)\frac{1}{\eta^2}(x_2-x_1)^2,
\end{align*}
where $K$ is some constant and where we used the fact that for $0\leq s \leq 1,$
\begin{multline}\label{ineq}
\Big|\exp\Big(is\Big(\frac{X_j-x_2}{h}\Big)\Big)
-\exp\Big(is\Big(\frac{X_j-x_1}{h}\Big)\Big)\Big|\\
\leq \Big|\cos\Big(s\Big(\frac{X_j-x_2}{h}\Big)\Big)-\cos\Big(s\Big(\frac{X_j-x_1}{h}\Big)\Big)\Big|\\
+\Big|\sin\Big(s\Big(\frac{X_j-x_2}{h}\Big)\Big)-\sin\Big(s\Big(\frac{X_j-x_1}{h}\Big)\Big)\Big|
\leq \frac{2s}{h}|x_1-x_2|,
\end{multline}
which follows by converting the differences of sines and cosines into products and using the fact that $|\sin x|\leq |x|.$ Consequently, the process $a_n(R_n^{(2)}-\ex [R_n^{(2)}])$
is tight.

To prove tightness of the process $a_n(R_n^{(3)}-\ex
[R_n^{(3)}]),$ we first introduce the function $u$, given by
\begin{equation}\label{ufunction}
u(y)=
\frac{C|y|^{\lambda_0} \exp(-|y|^\lambda/\mu)}{\phi_k(y)}-1.
\end{equation}
By Condition \ref{condk} this function is bounded on
$\mathbb{R}\negmedspace\setminus \negmedspace (-\delta,\delta),$
where $\delta$ is an arbitrary positive number. Moreover, by
\eqref{chk} the function $xu(x)$ is also bounded and both
functions vanish at plus and minus infinity. It follows that
$(s/h)u(s/h)$ is bounded and tends to zero for all fixed $s$
with $|s| \geq \epsilon$ as $h \to 0$.

Using the function $u,$ rewrite $R_{n,j}^{(3)}(x)$ as follows
\begin{align*}
&R_{n,j}^{(3)}(x)=\frac{1}{2\pi h}
\Big(\int_{-1}^{-\epsilon}+\int_{\epsilon}^1\Big)
\exp\Big(is\Big(\frac{X_j-x}{h}\Big)\Big)\phi_w(s)\nonumber\\
&\quad\quad\quad\quad\times\Big(\frac{1}{\phi_k(s/h)}-\frac{1}{C}\Big(\frac{|s|}{h}\Big)^{-\lambda_0}
\exp(|s|^\lambda/(\mu h^\lambda))\Big) ds\\
&=
\frac{1}{2\pi h}
\Big(\int_{-1}^{-\epsilon}+\int_{\epsilon}^1\Big)
\exp\Big(is\Big(\frac{X_j-x}{h}\Big)\Big)\phi_w(s)\nonumber\\
&\quad\quad\quad\quad\times
\frac{1}{C}\Big(\frac{|s|}{h}\Big)^{-\lambda_0}
\exp(|s|^\lambda/(\mu h^\lambda))u(s/h)  ds.
\end{align*}
Next note that, as above, for positive $\eta$ we have by
\eqref{ineq} that
\begin{align*}
P\Big(a_n|R_n^{(3)}&(x_2)-\ex [R_n^{(3)}(x_2)] -(R_n^{(3)}(x_1)-\ex [R_n^{(3)}(x_1)])|\geq\eta\Big)\\
&\leq
\frac{a_n^2}{\eta^2}\,\frac{1}{n}\, \ex [(R_{n,1}^{(3)}(x_2)-R_{n,1}^{(3)}(x_1))^2]\\
&\leq \frac{a_n^2}{\eta^2}\,{1\over 4\pi^2h^2}\,\frac{1}{n}
\Big(\Big(\int_{-1}^{-\epsilon}+\int_{\epsilon}^1\Big)\phi_w(s)\\
&\quad\quad\quad\quad
\times\frac{1}{C}\Big(\frac{|s|}{h}\Big)^{-\lambda_0}
\exp(|s|^\lambda/(\mu h^\lambda))\frac{s}{h}\,u\Big(\frac{s}{h}\Big)ds\Big)^2(x_2-x_1)^2\\
&= o(1)\frac{1}{\eta^2}(x_2-x_1)^2
\end{align*}
and hence $a_n(R_n^{(3)}-\ex [R_n^{(3)}])$ is tight. By
Prohorov's theorem each of the three processes now converges
weakly to the zero process. Since the convergence in distribution
to a constant entails convergence to the same constant in
probability, this concludes the proof of the lemma.
\end{proof}
Finally, we combine the obtained results to prove Theorem
\ref{thm:sup}.
\begin{proof}[Proof of Theorem \ref{thm:sup}]
The proof is immediate from Lemmas \ref{limsup}--\ref{lemma3} just
proved, the fact that by (\ref{eu})
\begin{align*}
 a_n\, \frac{1}{\pi C}&\,h^{\lambda_0-1}\int_\epsilon^1 \phi_w(s)
s^{-\lambda_0} \exp(s^\lambda/(\mu h^\lambda))ds  \frac{1}{\sqrt{n}} \\
&\sim
\frac{A}{\pi C}\,\Big({\mu \over
\lambda}\Big)^{1+\alpha}\Gamma(\alpha+1),
\end{align*}
and Theorems 4.1 and 5.1 of
\citet{billingsley}.
\end{proof}

\end{document}